\documentclass[12pt,a4paper]{article}
\usepackage[utf8]{inputenc}
\usepackage[english]{babel}
\usepackage{lmodern}
\usepackage{amsmath}
\usepackage{amssymb}
\usepackage{amsthm}
\usepackage{bm}
\usepackage[unicode, colorlinks = true]{hyperref}
\usepackage{color}
\usepackage[affil-it]{authblk}
\usepackage{enumerate}

\def\p{\phi}
\def\b{\beta}

\def\ap{\mu}
\def\bp{\nu}

\def\e{\varepsilon}
\def\g{\gamma}
\def\a{\alpha}
\def\d{\delta}

\def\s{\sigma}
\def\l{\lambda}

\def\P{\mathbb{P}}

\def\ind{\bm{1}}
\newcommand{\chgt}[1]{\textcolor{black}{#1}}
\newcommand{\abs}[1]{\left \vert {#1} \right \vert}
\newcommand{\brc}[1]{\left \lbrace {#1} \right \rbrace}
\newcommand{\CC}{\mathrm{C}}
\newcommand{\dd}{\,\mathrm{d}} 

\newcommand{\E}{\mathbb{E}}
\newcommand{\eexp}[1]{\mathrm{e}^{#1}}
\newcommand{\HH}{\mathrm{H}}

\newcommand{\Hao}{\HH_{\alpha}^o}

\newcommand{\I}{\mathbb{I}}
\newcommand{\icon}{\to \infty}
\newcommand{\iid}{\mbox{i.i.d.\frenchspacing}}
\newcommand{\Lp}{\mathcal{L}_p}
\newcommand{\Lpi}{\mathcal{L}_{p,\infty}}
\newcommand{\Lpio}{\mathcal{L}^{o}_{p,\infty}}
\newcommand{\Lr}{\mathcal{L}_r}
\newcommand{\Lrio}{\mathcal{L}^{o}_{r,\infty}}

\newcommand{\N}{\mathbb{N}}
\newcommand{\norm}[1]{\left \Vert {#1} \right \Vert}
\newcommand{\pcon}{\xrightarrow[n \to \infty]{\mathrm{P}}}  
\newcommand{\R}{\mathbb{R}}
\newcommand{\RV}{\mathrm{RV}}
\newcommand{\spcon}[1]{\xrightarrow[n \to \infty]{\;\mathrm{#1}\;}} 

\newcommand{\wh}{\widehat}
\newcommand{\wt}{\widetilde}

\theoremstyle{plain}
\newtheorem{theorem}{Theorem}[section]
\newtheorem{lemma}[theorem]{Lemma}

\theoremstyle{remark}

\theoremstyle{definition}
\newtheorem{definition}[theorem]{Definition}

\def\Wpl{W_{n}}

\title{\textbf{Testing epidemic change in nearly nonstationary process with
  statistics based on residuals}}

\author[1]{J. Markevi\v{c}i\={u}t\.{e}%
  \thanks{Electronic address: \texttt{jurgita.markeviciute@mif.vu.lt}; Corresponding author}}

\author[1]{A. Ra\v{c}kauskas%
  \thanks{Electronic address: \texttt{alfredas.rackauskas@mif.vu.lt}}}
  
\author[2]{Ch. Suquet%
  \thanks{Electronic address: \texttt{Charles.Suquet@math.univ-lille1.fr}}}

\affil[1]{Faculty of Mathematics and Informatics, Vilnius University\\
Naugarduko str. 24, LT-03225 Vilnius, Lithuania}

\affil[2]{Laboratoire P. Painlev\'{e}, UMR 8524 CNRS Universit\'{e} Lille I\\
B\^{a}t. M2, Cit\'{e} Scientifique, F-59655 Villeneuve d'Ascq Cedex, France}

\date{Dated: \today}

\begin{document}

\maketitle

\begin{abstract}
  We study an epidemic type change in innovations of a
  first order autoregressive process $ y_{n,k} = \p_n y_{n,k-1} +
  \e_{k} + a_{n,k}$, where $\p_n$ is either a constant in $(-1,1)$ or a
  sequence in $(0,1)$, converging to 1. For $k$ inside some unknown
  interval $\mathbb{I}_n^\ast=(k^\ast,k^\ast+\ell^\ast]$,
  $a_{n,k}=a_n$ while $a_{n,k}=0$ for $k$ outside
  $\mathbb{I}_n^\ast$. When $a_n\neq 0$, we have an epidemic deviation
  from the usual (zero) mean of innovations. Since innovations are not
  observed, we build uniform increments statistics on residuals
  $(\wh{\e}_k)$ of the process $y_{n,k}$. We assume that innovations
  $(\e_k)$ are regularly varying with index $p \ge 2$ or satisfies
  integrability condition $\lim_{t \to \infty} t^p \P(\abs{\e_1} > t)
  = 0$ for $p > 2$ and $\E\e_k^2 < \infty$ for $p=2$. We find the
  limit distributions of the tests under no change and prove 
  consistency under short epidemics that is $\ell^\ast=O(n^\beta)$ for
  some $0<\beta\le 1/2$.
\end{abstract}



\section{Introduction}

Suppose we are given a sequence of observations $(\e_k)$ that are
assumed to be independent identically distributed (i.i.d.) random
variables with zero mean, except, maybe, for a short interval, where
the mean of the corresponding observations is a nonzero constant. Such
model can be interpreted as an epidemic one. The nonzero mean
corresponds to an epidemic deviation from the usual state. The length
of the interval describes the duration of the epidemic.  The question
is how to decide whether such an interval is present.  To the best of
our knowledge, this kind of problem was formulated for the first time
by Levin and Kline \cite{Levin-Kline:1985} in the context of abortion
epidemiology.  Simultaneously, epidemic type models were introduced by
Commenges, Seal and Pinatel \cite{Commenges-Seal-Pinatel:1986} in
connection with experimental neurophysiology. Models with an epidemic
type change in the mean were also used for detecting changed segments
in non-coding DNA sequences \cite{Avery:1999} and for studying
structural breaks in econometric contexts \cite{Bro-Tsu:1987}.  Levin
and Kline \cite{Levin-Kline:1985} proposed the test statistic
$\max_{1\le\ell\le n} \max_{0\le k\le
  n-\ell}(\sum_{j=k+1}^{k+\ell}\e_j-\ell \delta/2)$, where $\delta>0$
represents the smallest increment in the mean which is sufficiently
important to be detected.  Another type of statistics can be
constructed by normalizing the sums $\sum_{j=k+1}^{k+\ell}\e_j$
according to our guess about the length of the epidemic state. In this
way we arrive at the following multiscale type statistic
\[
T_n({\e}_1, \dots, {\e}_n)= \max_{1 \le \ell \le n} \ell^{-\alpha}
\max_{1\le k\le n-\ell}\Big(\sum_{j=k+1}^{k+\ell}{\e}_j\Big),
\]
where $0\le \alpha\le 1$.  Large values of this statistic indicate the
presence of an epidemic state. We refer
to~\cite{Rackauskas-Suquet:2004a} for various asymptotic results for 
this type of statistics. Yet another type of statistic based on ranks and
signs of observations where suggested by Gombay
\cite{Gombay:1994}. She also pointed out that despite the fact that
the epidemic model can be formulated as multiple change-point model,
tests constructed with account of a particular form of changes may
have bigger power.

The problem that we are concerned with in this paper is addressed to a
situation where the sequence $(\e_k)$ appears as innovations in a
certain time series model. So that we cannot observe $(\e_k)$
directly. What usually we have at hands are residuals $(\wh{\e}_k)$
obtained by an estimation procedure of the model under
consideration. This suggests that testing for an epidemic state in the
sequence $(\e_k)$ can be based on residuals $(\wh{\e}_k)$.
To be more precise, assume we are given a sample $y_{n, 1}, \dots,
y_{n,n}$ for a fixed $n$, generated from the first order
 autoregressive process
\begin{align}
  \label{eq:1}
  y_{n,k} = \p_n y_{n,k-1} + \e_{k} + a_{n,k}, \quad k=1, \dots, n,
  \quad n\ge 1, \quad y_{n,0} = 0,
\end{align}
where the unknown coefficient $\p_n$ is either a constant $\p$ in
$(-1, 1)$ or a sequence of constants $\p_n\in(0,1)$ and $\p_n$ tends to $1$, as $ n \icon $.
{{The innovations $ (\e_{k} , k \le n)$ are unobservable, centered, at
least square integrable random variables.}}
In what follows we denote
\begin{equation}\label{def-gam}
  \g_n:=n(1-\p_n),
\end{equation}
and assume throughout the paper that
$\lim_{n\to\infty}\g_n=\infty$. 
When all the $a_{n,k}$ are null and
$\p_n$ tends to 1, the process $y_{n,k}$ is called \emph{nearly nonstationary}. We refer to Giraitis and
Philips~\cite{Giraitis-Phillips:2004} for a study of the asymptotic
behaviour of such a process.

The aim of this paper is to propose tests for the null hypothesis
$$
H_0:\quad a_{n,1}=\cdots=a_{n,n}=0
$$
against the epidemic alternative:
\begin{align*}
  H_A:\ \ &\textrm {there exist} \quad 0\le k^*_n<n, \quad 1 \le m^*_n\le n \quad \textrm{such that}\\
  & a_{n,k}=a_n \neq 0 \quad \textrm{for}\quad k\in \I_n^*\quad
  \textrm{whereas}\quad a_{n,k}=0\quad \textrm{for}\quad k\not\in
  \I^*_n,
\end{align*}
where $\I_n^*=\{k_n^*+1, \dots, m_n^*\}$. The value $a_n$ during the
period $\I_n^*$ is interpreted as an epidemic deviation from the usual
(zero) mean of innovations and $\ell_n^*=m_n^*-k_n^*$ is the duration
of the epidemic state.

Set for $\alpha\in [0, 1)$ and any real numbers $x_1, \dots, x_n$:
\begin{align*}
T_{\a, n}({x}_1, \dots, {x}_n)
&= \max_{1 \le \ell \le n} \ell^{-\alpha}
\max_{1\le k\le n-\ell}\abs{\sum_{j=k+1}^{k+\ell}{x}_j-
\frac{\ell}{n}\sum_{j=1}^n x_j}.
\end{align*}
Denote
\begin{equation}\label{T0}
  {T}_{\a, n}=T_{\a, n}({\e}_1, \dots, {\e}_n),
\end{equation}
and
\begin{equation}\label{T}
  \wh{T}_{\a, n}=T_{\a, n}(\wh{\e}_1, \dots, \wh{\e}_n),
\end{equation}
where $(\wh{\e}_k)$ are residuals of the model \eqref{eq:1} defined by
$$
\wh{\e}_{k}=y_{n,k}-\wh{\p}_n y_{n, k-1}, \quad k=1, \dots, n,
$$
and $\wh{\p}_n$ is the least square estimator of $\p_n$:
\begin{equation}\label{deflqe}
\wh{\p}_n=\frac{\sum_{k=1}^ny_{n,k}y_{n, k-1}}{\sum_{k=1}^ny_{n,k-1}^2}.
\end{equation}

Roughly speaking, under $H_A$, the probability of detection of the
epidemic is an increasing function of the amplitude of the jump
$\abs{a_n}$ and of the length of the epidemic interval $\ell^*$. If we
take $a_n=a$ constant and assume that $\ell^*=\theta n^\beta$, for
some $0<\beta\le 1$, then the choice $\a=0$ in the definition of
$T_{\a,n}$ leads to the classical CUSUM procedure which allows to
detect epidemic of length $\ell^*=\theta n^\beta$ when
$\beta>1/2$. The asymptotic behavior of $T_{0,n}$ and $\wh{T}_{0,n}$
are deduced from a functional central limit theorem in the space
$\CC[0,1]$. The main interest of the statistics $\wh{T}_{\a,n}$ is
that the possibility to choose $\a>0$, subject to some additional
integrability condition on the innovations, allows the detection of
shorter epidemics, of length $\ell^*=\theta n^{\beta}$ with
$\beta<1/2$. For a study of the epidemic detection in an \iid\ sample
via H\"olderian techniques, we refer
to~\cite{Rackauskas-Suquet:2004a}.

We investigate the limit behavior of $\wh{T}_{\alpha,n}$ for two
classes of innovations $(\e_k)$.

\begin{definition} Let $p>0$. We say that a random variable $X$
  belongs to the class
  \begin{itemize}
  \item $\Lpi$, if $\displaystyle \sup_{t>0}t^p \P(\abs{X}>t)<\infty$,
  \item $\Lpio$, if $\displaystyle \lim_{t\to\infty} t^p\P(|X|>t)=0$,
  \item $\Lp$, if $\E\abs{X}^p<\infty$.
  \end{itemize}
\end{definition}
It is well known that for $0<r<p$, \(\Lp \subset \Lpio \subset
\Lpi\subset\Lr\).

\begin{definition}\label{Def-RVp}
  The random variable $X$ is regularly varying with index $p >0$
  (denoted $X\in \RV_{p}$) if there exists a slowly varying function
  $L$ such that the distribution function $F(t)=P(X\le t)$ satisfies
  the tail balance condition
$$
F(-x) \sim b L(x)x^{-p}\quad \textrm{and}\quad 1 -F(x) \sim a
L(x)x^{-p}, \quad \textrm{as}\quad x\to\infty,
$$
where $a, b \in (0, 1)$ and $a +b = 1$.
\end{definition}
We refer to \cite{BGT:1987} for an encyclopedic treatment of regular
variation.  We note that if $0<r<p$, then $\RV_p \subset
\Lrio$. Moreover, if $L(x)\to 0$ as $x\to \infty$, then $\RV_p \subset
\Lpio$. Further, if $\e_1\in \RV_{p}$ then define
\begin{equation}\label{a_n}
  b_n=\inf\{x>0: \P(|\e_1|\le x)\ge 1-1/n\}.
\end{equation}
It easily follows from tail condition that there is a slowly varying
function $v(n), n\in \N$, such that
\begin{equation}\label{a_n:1}
  b_n\ \sim\ n^{1/p}v(n) \quad \textrm{as} \quad n \icon.
\end{equation}

Throughout the paper, $\spcon{\mathcal{D}}$ means convergence in distribution.

The paper is organized as follows. In Section 2 we establish limits of
distributions of the test statistics under null hypothesis. Section 3
contains consistency analysis of the test statistics. All the proofs
of technical intermediate results are detailed in the appendix.


\section{Limit behavior of test statistics under null hypothesis}


From now on, for any $p\ge 2$, we set
\begin{align}\label{def-ap}
  \a_p &:= \frac{1}{2} - \frac{1}{p}.
\end{align}
Since we assume $H_0$ true throughout this section, the data generating
process $y_{n,k}$ is given by
\begin{align}\label{090314a}
  y_{n,k} = \p_n y_{n,k-1} + \e_{k}, \quad k=1, \dots, n,
  \quad n\ge 1, \quad y_{n,0} = 0.
\end{align}

The following lemma is the key to connect the asymptotic behavior of
$\wh{T}_{\alpha,n}$ to the one of $T_{\alpha, n}$. Its proof being quite long and technical is deferred to the
appendix (section A).

\begin{lemma}\label{Lem-181113A} 
  Assume that the $\e_i$'s are \iid\ random variables and the
  $y_{n,k}$'s satisfy~\eqref{090314a}.
  
\begin{enumerate}[a)]
\item
If $\e_1\in \RV_p$ for some $p\ge 2$ and
\begin{align}\label{090314b}
\lim_{n\to \infty}n(1-\p_n)=\infty.
\end{align}
then for any $\a \in (\a_p,1]$ with $\a_p$ defined by~\eqref{def-ap}, 
  \begin{align}
    (1-\p_n)T_{\alpha, n}(y_{n,0},\ldots,y_{n,n-1})
    &= 
      T_{\alpha, n}(\e_1, \dots, \e_n)+O_P(b_n)\label{181113a}\\
      T_{\alpha, n}(\wh{\e}_1,\ldots, \wh{\e}_n) 
    &=     
      T_{\alpha, n}(\e_1,\ldots,\e_n)+o_P(b_n).  \label{181113b}
  \end{align}
where $b_n$ is defined by~\eqref{a_n}.
\item If $\e_1\in\mathcal{L}_2$ (then $p=2$) or $\e_1\in \Lpio$ for
  some $p>2$, and
\begin{align}\label{160314a}
\liminf_{n\to \infty}n^{1-\delta}(1-\p_n)>0\quad\text{for some $\delta>0$,}
\end{align}
then for any $\a \in[0,\a_p]$,
\begin{align}
    (1-\p_n)T_{\alpha, n}(y_{n,0},\ldots,y_{n,n-1})
   & = 
      T_{\alpha, n}(\e_1, \dots, \e_n)+O_P(n^{1/2-\alpha})\label{181113A}\\
T_{\alpha, n}(\wh{\e}_1,\ldots, \wh{\e}_n) 
    &=     
      T_{\alpha, n}(\e_1,\ldots,\e_n)+o_P(n^{1/2-\a}).\label{181113B}
  \end{align}
\end{enumerate} 
\end{lemma}

\begin{theorem}\label{Th301013a} 
  Assume that the $\e_i$'s are mean zero \iid\ random
  variables in $\mathcal{L}_2$ (then $p=2$) or in $\Lpio$ for some $p>2$, and
  that the $y_{ n,k}$'s satisfy~\eqref{090314a} with $\p_n $
  satisfying~\eqref{160314a}. Then for any $\a \in[0,\a_p]$,
  \begin{align}
    \label{301013c}
    n^{-1/2+\a}\s^{-1} \wh{T}_{\a, n} \spcon{\mathcal{D}}
    T_{\a,\infty}(W):=\max_{0<h<1}h^{-\alpha}\max_{0\le t\le 1-h}|W_{t+h}-W_t-hW_1|,
  \end{align}
  where $\sigma^2=\E\e_1^2$ and $W=\{W_t, 0\le t\le 1\}$ is a standard
  Brownian motion.
\end{theorem}

\begin{proof}
  Let $\Hao$ be the set of continuous functions $f:[0,1]\to\R$ such
  that $\omega_\a(f,\delta):=\sup\{\abs{f(t)-f(s)}\abs{t-s}^{-\a} :
  0<\abs{t-s}\le \delta\}$ tends to zero as $\delta$ tends to
  zero. For $0\le \a <1$, $\Hao$ is equipped with the $\alpha$-H\"older norm $\norm{f}_\a:=\abs{f(0)}+\omega_\a(f,1)$. It is known
  that if $f$ is a polygonal line, then the supremum in the definition
  of $\omega_\a(f,1)$ is reached at two vertices, see
  e.g.~\cite[Lem. A.2]{MRS:2012a}. Hence $n^\a T_{\alpha,
n}(\e_1,\ldots,\e_n)$ is \emph{exactly} the $\a$-H\"older norm of
  the random \emph{polygonal} function $W_n - W_n(1)\mathrm{Id}$,
  where $\mathrm{Id}$ is the identity function on $[0,1]$ and
\begin{align}
  \label{myq:3a}
  \Wpl(t) = \sum_{k=1}^{[nt]} \e_k + \{nt\} \e_{[nt]+1}, \quad t \in
  [0,1],
\end{align}
where $\{nt\}$ is the fractional part of $nt$. Since $\e_1\in\Lpio$,
$n^{-1/2}\sigma^{-1}\Wpl$ converges in distribution to the standard
Brownian motion $W$ in $\HH_{\a_p}^o$, where $\a_p=1/2-1/p$,
see~\cite{Rackauskas-Suquet:2004}. By topological inclusions of
H\"older spaces, the same convergence holds in any H\"older space of
exponent $0<\a<\a_p$ if $\a_p>0$. In the special case $\a=0$, $\Hao$
is isomorphic to $\CC[0,1]$ and the convergence of
$n^{-1/2}\sigma^{-1}\Wpl$ is simply the classical invariance principle
from Donsker-Prokhorov. Since the linear operator $B: f\mapsto f -
f(1)\mathrm{Id}$ is continuous on $\Hao$, $B(n^{-1/2}\sigma^{-1}\Wpl)$
converges in distribution on $\Hao$ to $B(W)$, for $0\le \a \le
\a_p$. Hence by continuous maping,
$n^{-1/2+\a}T_{\alpha,n}(\e_1,\ldots,\e_n)=
\norm{B(n^{-1/2}\sigma^{-1}\Wpl)}_\a$ converges in distribution to
$\norm{B(W)}_\a=T_{\a,\infty}(W)$. In view of~\eqref{181113B} in
Lemma~\ref{Lem-181113A}, $n^{-1/2+\a}\s^{-1} \wh{T}_{\a, n}$ converges
in distribution to the same limit.
\end{proof}

In the case of $p$-regularly varying innovations with $p>2$, in view of the the
inclusion $\RV_p\subset \Lrio$ for $r<p$, the limit distribution of
$\wh{T}_{\a,n}$ is given by Theorem~\ref{Th301013a}, where $p$ is
replaced by $2\le r<p$ and $\a_p$ by $\a_r=1/2-1/r$, the
choice of an appropriate $r$ depending on the rate of
$\g_n$. Moreover, if the slowly varying function $v$ of
Definition~\ref{Def-RVp} tends to zero at infinity then $\RV_p\subset
\Lrio$ and Theorem~\ref{Th301013a} applies directly.

\begin{theorem}\label{Th241112b}
  Assume that the $\e_i$'s are mean zero \iid\ random variables in
  $\RV_{p}$ for some $p\ge 2$ and the $y_{n,k}$'s
  satisfy~\eqref{090314a} with $\p_n $
  satisfying~\eqref{090314b}. Then for any $\a \in(\a_p, 1]$,
\begin{equation}
  \label{Th241112b:eq:1}
  b_n^{-1}\wh{T}_{\a, n}\spcon{\mathcal{D}} T_{p},
\end{equation}
where $T_{p}$ is a random variable with Fréchet distribution $\P(T_p\le
x)=\exp(-x^{-p})$, $x>0$.
\end{theorem}

\begin{proof}
  By Theorem 1.1. in \cite{MR:2010}, if the innovations $\e_i$ are
  \iid\ and in $\RV_p$, then
  \begin{equation}\label{051113a}
    b_n^{-1}T_{\a,n}(\e_1,\ldots,\e_n)\spcon{\mathcal{D}} T_{p},
  \end{equation}
  so \eqref{Th241112b:eq:1} obviously follows from \eqref{181113b} in
  lemma~\ref{Lem-181113A}.
\end{proof}


\section{Consistency of test statistics}


In this section we investigate the consistency of the test statistics
$\wh{T}_{\a, n}$. 
So we are given a sample $(y_{n,k}, k = 1, . . . , n)$ generated from the first order
autoregressive process with epidemic drift
\[
y_{n,k} = \p_n y_{n,k-1} + \e_{k} + a_{n}\ind_{\I_n^*}(k), \quad k=1, \dots, n,
\quad n\ge 1, \quad y_{n,0} = 0
\]
where $\I_n^*=\{k_n^*+1,\dots,k_n^*+\ell_n^*\}$.
Let us introduce
\begin{align}
  \label{050213d}
  \tau_{n,k} =\sum_{j=1}^k
  \p_n^{k-j} a_{n,j}
\end{align}
and
\begin{align}
  \label{091113a}
  z_{n,k} = y_{n,k} - \tau_{n,k}, \quad k=0,1,\ldots ,n.
\end{align}
Noting that $\tau_{n,k} - \p_n\tau_{n,k-1} =a_{n,k}$, we can recast
the model giving the $y_{n,k}$'s as
\begin{align*}
  y_{n,k}-\tau_{n,k} &= \p_n (y_{n,k-1}-\tau_{n,k-1}) + \e_{k}.
\end{align*}
It follows that \emph{if $(y_{n,k})$ satisfies $H_A$, then $(z_{n,k})$ is an
$AR(1)$ process satisfying $H_0$}.

To exploit this feature, we can express the residuals in the following
way.
\begin{align*}
  \wh{\e}_k = y_{n,k} - \wh{\p}_n y_{n,k-1} &=
  \p_n y_{n,k-1} + \e_k + a_{n,k}  - \wh{\p}_n y_{n,k-1}\\
  &=
  (\p_n - \wh{\p}_n)(z_{n,k-1} + \tau_{n,k-1}) + \e_k + a_{n,k}\\
  &= a_{n,k} + (\p_n - \wh{\p}_n)\tau_{n,k-1} +(\p_n -
  \wh{\p}_n)z_{n,k-1} + \e_k.
\end{align*}
Triangle inequality applied to $\wh{T}_{\a, n}=T_{\a,
  n}(\wh{\e}_{1},\ldots , \wh{\e}_{n})$ leads to
\begin{align}
  \label{101113a}
  \wh{T}_{\a, n} &\ge T_{\a,n}(a_{n,1},\ldots , a_{n,n})
  - \abs{\wh{\p}_n - \p_n} T_{\a,n}(\tau_{n,0},\ldots , \tau_{n,n-1}) \nonumber \\
  &\phantom{\ge} {} - \abs{\wh{\p}_n - \p_n} T_{\a,n}(z_{n,0},\ldots ,
  z_{n,n-1}) - T_{\a,n}(\e_{1},\ldots , \e_{n}).
\end{align}

Now, to obtain the consistency of our statistics $\wh{T}_{\a, n}$, it
suffices to prove that with the normalization already used under
$H_0$, all the random terms in the above lower bound are negligible in
probability when compared with the deterministic term
$T_{\a,n}(a_{n,1},\ldots , a_{n,n})$ which has to tend to infinity. In
this way, it is convenient to replace $T_{\a,n}(a_{n,1},\ldots ,
a_{n,n})$ by the following lower bound, assuming without loss of
generality that $\ell_n^\ast\le n/2$ (recall we are looking for short
epidemics).
\begin{align}
  T_{\a,n}(a_{n,1},\ldots , a_{n,n}) &= \max_{1 \le \ell \le n}
  \ell^{-\a} \max_{1 \le k \le n-\ell} \abs{\sum_{j=k+1}^{k+\ell} a_n
    \ind_{\I_n^*}(j)
    - \frac{l}{n} \sum_{j=1}^n a_n \ind_{\I_n^*}(j)} \nonumber\\
  &\ge \abs{a_n} \ell^{*(1-\a)} \left(1-\frac{\ell^*}{n}\right) \ge
  \frac{1}{2} \abs{a_n} \ell^{*(1-\a)}.\label{111113a}
\end{align}

As in~\eqref{101113a}, both random terms $T_{\a,n}(z_{n,0},\ldots ,
z_{n,n-1})$ and $T_{\a,n}(\e_{1},\ldots , \e_{n})$ can be controlled
by the H\"olderian functional central limit theorems already used
under $H_0$, it remains to find suitable estimates for
$T_{\a,n}(\tau_{n,0},\ldots , \tau_{n,n-1})$ and $\abs{\wh{\p}_n -
  \p_n}$. This is provided by the following lemmas, whose the
quite technical proofs are postponed to Section B in the annex.

\begin{lemma}
  \label{L281213}
  Suppose that $k^* \ge \l n$ with some fixed $0 < \l <1 $. Assume
  that the innovations $\e_i$ of the process $(y_{n,k})$ defined by
  \eqref{eq:1} are square integrable and that $\g_n$ is increasing in
  $n$ or regularly varying. Then
 \[
    \abs{\wh{\p}_n - \p_n} = o_P(1-\p_n),
 \]
provided that 
\begin{align}
\label{281213A}
a_n^2\ell^\ast = o\left(n(1-\p_n)\right).
\end{align}
\end{lemma}

\begin{lemma}
  \label{L170413d}
  Under $H_A$, with $\tau_{n,k}$ defined by \eqref{050213d},
  \begin{align}
    \label{proo170413A}
    T_{\a,n}(\tau_{n,0},\ldots , \tau_{n,n-1}) \le
    \frac{5\abs{a_n}}{1-\p_n}\ell^{*(1-\a)}
  \end{align}
\end{lemma}

Now we are in a position to give our consistency results.
\begin{theorem}\label{Th111113a}
  Suppose that in the model defined by~\eqref{eq:1}, $\p_n$ is a
  constant $\p\in(-1,1)$ and that the $\e_i$ are in $\mathcal{L}_2$
  (then $p=2$) or in $\Lpio$ for some $p>2$. Assume that
  \begin{align}\label{H111113a}
    \ell^\ast a_n^2 = o(n)
  \end{align}
  and that for some $\a\in[0,\a_p]$,
  \begin{align}\label{H111113b}
    n^{-1/2+\a}\abs{a_n}{\ell^\ast}^{(1-\a)}\icon.
  \end{align}
  Then, under $H_A$,
  \begin{align*}
    n^{-1/2+\a} \wh{T}_{\a,n} \pcon \infty .
  \end{align*}
\end{theorem}
Of course the same result holds when the $\e_i$ are in $\RV_p$,
assuming that $\a<\a_p$ in~\eqref{H111113b}. It is worth noting here
that if $a_n$ is constant, then Theorem~\ref{Th111113a} allows
detection of short epidemics satisfying $n^\b = o(\ell^\ast)$ with
$\beta=(1/2-\a)/(1-\a)$. In particular, if the $\e_i$ have finite $p$
moments for every $p>0$, then epidemics of length $n^\b$ are
detectable for arbitrarily small $\b$. Since the proof of
Theorem~\ref{Th111113a} is a simple adaptation of the one of
Theorem~\ref{Th180413a} below, we omit it.

\begin{theorem}
  \label{Th180413a}
  Suppose that in the model defined by~\eqref{eq:1}, the $\e_i$'s are
  in $\mathcal{L}_2$ (then $p=2$) or in $\Lpio$ for some
  $p>2$. Suppose moreover that $\p_n\in(0,1)$, $\p_n\to 1$ and that
  $\g_n=n(1-\p_n)$ is non decreasing or regularly varying and
  satisfies~\eqref{160314a}.  Under $H_A$, assume that
\begin{align}
    \label{180413a}
    \wh{\p}_n-\p_n =o_P(1-\p_n),
  \end{align}
that  $\ell^* \icon$, $\ell^{*}=o(n)$ and for
  some $\a\in[0,\a_p]$,
  \[
  n^{-1/2+\a} \ell^{*(1-\a)}\abs{a_n} \xrightarrow[n \to \infty]{}
  \infty.
  \]
 Then
  \begin{align}
    \label{proo180413a}
    n^{-1/2+\a} \wh{T}_{\a,n} \pcon \infty.
  \end{align}
  Condition \eqref{180413a} holds in particular if 
  $a_n^2\ell^\ast=o(n(1-\p_n))$.
\end{theorem}

\begin{proof}
  As already seen in the proof of Theorem~\ref{Th301013a}, the
  membership of $\e_1$ in $\mathcal{L}_2$ or in $\Lpio$ implies the
  convergence in distribution of $n^{-1/2+\a}\s^{-1} T_{\alpha,
    n}(\e_1,\ldots,\e_n)$ to $T_{\a,\infty}(W)$ for every
  $\a\in[0,\a_p]$, with $\a_2=0$ in the $\mathcal{L}_2$ case, whence
  \begin{align}
    \label{060513b}
    n^{-1/2+\a} T_{\a,n}(\e_{1},\ldots , \e_{n}) = O_P(1).
  \end{align}
  Next, recalling that under $H_A$, $(z_{n,k})$ is a nearly non
  stationary process satisfying $H_0$, we know from~\eqref{181113A} in
  Lemma~\ref{Lem-181113A} that for every $\a\in[0,\a_p]$, 
  \begin{align}
    \label{060513c}
    n^{-1/2+\a}T_{\a,n}(z_{n,0},\ldots , z_{n,n-1}) = O_P(1/(1-\p_n)).
  \end{align}
  Looking back at~\eqref{101113a}--\eqref{proo170413A} and
  accounting~\eqref{060513b}, \eqref{060513c}, we obtain the lower
  bound
  \begin{align*}
    n^{-1/2+\a} \wh{T}_{\a, n} \ge \frac{1}{2} n^{-1/2+\a}\abs{a_n}
    \ell^{*(1-\a)} - \Delta_n,
  \end{align*}
  where
  \begin{align*}
    \Delta_n = 5n^{-1/2+\a}\abs{a_n}\ell^{*(1-\a)}\frac{\abs{\wh{\p}_n
        - \p_n}}{1-\p_n} + O_p\left(1+\frac{\abs{\wh{\p}_n -
          \p_n}}{1-\p_n}\right).
  \end{align*}
  It is clear from  this lower bound, that if $
  n^{-1/2+\a}\abs{a_n}\ell^{*(1-\a)}$ tends to infinity and
  $\wh{\p}_n-\p_n =o_P(1-\p_n)$, then $n^{-1/2+\a} \wh{T}_{\a,n}$
  tends in probability to infinity. A concrete condition on
  $a_n^2\ell^\ast$ to have $\wh{\p}_n-\p_n =o_P(1-\p_n)$ is given by
  Lemma~\ref{L281213}.
\end{proof}

We complete Theorem~\ref{Th180413a} by treating the
case $\a\in(\a_p,1]$ for $\e_i\in\RV_p$.

\begin{theorem}
  \label{Th140513a}
  Suppose that in the model defined by~\eqref{eq:1}, the $\e_i$'s are
  in $\RV_p$ for some $p\ge 2$. Suppose moreover that $\p_n\in(0,1)$,
  $\p_n\to 1$ and that $\g_n=n(1-\p_n)$ is non decreasing or regularly
  varying, tends to infinity.
Under $H_A$, assume that $\ell^*
\icon$, $\ell^{*}=o(n)$ and
  \[
  b_n^{-1} \ell^{*(1-\a)}\abs{a_n} \xrightarrow[n \to \infty]{}
  \infty
  \]
  with $b_n$ defined by~\eqref{a_n}.  If in addition $\wh{\p}_n$
  satisfies \eqref{180413a}, then
  \begin{align}
    \label{proo140513a}
    b_n^{-1} \wh{T}_{\a,n} \pcon \infty.
  \end{align}
\end{theorem}

\begin{proof}
  The proof relies again on the lower bound~\eqref{101113a}, where the
  estimates \eqref{111113a} and \eqref{proo170413A} for the
  deterministic terms $T_{\a,n}(a_{n,1},\ldots , a_{n,n})$ and
  $T_{\a,n}(\tau_{n,0},\ldots , \tau_{n,n-1})$ remain valid. For the
  control of the random terms $T_{\a,n}(z_{n,0},\ldots , z_{n,n-1})$,
  $T_{\a,n}(\e_{n,0},\ldots , \e_{n,n-1})$, as the
  $\e_i$'s are in $\RV_p$, Theorem~1.1. in~\cite{MR:2010} provides
  \begin{align}
    \label{140513a}
    b_n^{-1} T_{\a,n}(\e_{1},\ldots , \e_{n}) = O_P(1).
  \end{align}
  From~\eqref{140513a}, \eqref{181113a} in Lemma~\ref{Lem-181113A} applied to the
  process $(z_{n,k})$ which satisfies $H_0$, and~\eqref{180413a},  we have
\begin{align*}
\abs{\wh{\p}_n-\p_n}T_{\a,n}(z_{n,0},\ldots , z_{n,n-1}) = o_P(b_n).
\end{align*}
Collecting all previous estimates we obtain from~\eqref{101113a}
\[
  \wh{T}_{\a, n} 
\ge
\frac{1}{2}\abs{a_n}{\ell^\ast}^{(1-\a)} - 
o_P\left(\abs{a_n}{\ell^\ast}^{(1-\a)}\right)   - o_p(b_n) - O_P(b_n).
\]
Clearly now, in order that $ b_n^{-1} \wh{T}_{\a,n}$ tends to infinity
in probability, it suffices that
$b_n^{-1}\abs{a_n}{\ell^\ast}^{(1-\a)}$ tends to infinity.
\end{proof}

\appendix

\section{Proof of Lemma \ref{Lem-181113A}}
\label{proof-Lem-181113A}

\subsection{Some useful inequalities}
The forthcoming proof of lemma~\ref{Lem-181113A}, which is an
essential tool in proving theorems~\ref{Th241112b} and~\ref{Th140513a}, 
exploits intensively  the following version of H\'ajek-R\'enyi
inequality.
\begin{lemma}\label{H_R} 
  For each $n\ge 1$ let $(S_{n,k}, 1\le k\le n)$ be a sequence of
  random variables, let $(\ap_{n,k}, 1\le k \le n)$ , $(\bp_{n,k}, 1\le
  k\le n)$ be sequences of nonnegative real numbers and let $r\ge 2$.
  If there exists $c>0$ such that for any $1\le m\le n$ and any $\d>0$
\[
\P\Big(\max_{k\le m}|S_{n,k}|\ge \d\Big)\le
c\d^{-r}\Big[\Big(\sum_{k=1}^m \ap_{n,k}\Big)^{r/2}+\sum_{k=1}^m \bp_{n,k}\Big]
\]
then for any sequence $(\beta_{n,k}, 1\le k\le n)$ such that
$0<\beta_{n,1}\le \cdots\le \beta_{n,n}$ and any $\d>0$ it holds
\[
\P\Big(\max_{k\le n}\beta_{n,k}^{-1}|S_{n,k}|\ge \d\Big)\le 
2^{1+r/2}c\d^{-r}\Big[\Big(\sum_{k=1}^n \beta_{n,k}^{-2}\ap_{n,k}\Big)^{r/2}
+\sum_{k=1}^n\beta_{n,k}^{-r}\bp_{n,k}\Big].
\]

\end{lemma}

\begin{proof}
  The proof is based on the idea of the proof of Theorem 1.1 in
  Fazekas and Klesov~\cite{Fazekas-Klesov:2000}. Without loss of
  generality we can assume that $\beta_{n,1}=1$.  Set for $i\ge 0$,
\[
A_i:=\{m: 1\le m\le n \quad \textrm{and}\quad 2^i\le \beta_{n,m}^r<2^{i+1}\}
\]
and let $I=\max\{i: A_i\not=\emptyset\}.$ Let $m_i$ be a maximal
element of the set $A_i$ and put $m_i=m_{i-1}$ if the set $A_i$ is
empty.
Then we have
\begin{align*}
\P\left(\max_{1\le k\le n}\beta_{n,k}^{-1}|S_{n,k}|>\d\right)
&\le
\P\left(\max_{0\le i\le I}2^{-i/r}\max_{k\in A_i}|S_{n,k}|\ge \d\right)\\ 
&\le 
\sum_{i=0}^I \P\left(\max_{k\le m_i}|S_{n,k}|\ge \d 2^{i/r}\right)\\
&\le 
c\d^{-r}\sum_{i=0}^I 2^{-i}\Big[\Big(\sum_{k=1}^{m_i} \ap_{n,k}\Big)^{r/2}+
\sum_{k=1}^{m_i} \bp_{n,k}\Big]
\end{align*}
By comparison of $\ell^p$ and $\ell^1$ norms on $\R^{1+I}$ with $p=r/2\ge 1$,
\[
\sum_{i=0}^I 2^{-i}\Big(\sum_{k=1}^{m_i} \ap_{n,k}\Big)^{r/2}\le 
\Big(\sum_{i=0}^I 2^{-2i/r}\sum_{k=1}^{m_i} \ap_{n,k}\Big)^{r/2},
\]
from which we deduce 
\[
\P\left(\max_{1\le k\le n}\beta_k^{-1}|S_{n,k}|>\d\right)
\le  
c\d^{-r}\Big[\Big(\sum_{i=0}^I2^{-2i/r}\sum_{k=1}^{m_i} \ap_{n,k}\Big)^{r/2}+
\sum_{i=0}^I2^{-i}\sum_{k=1}^{m_i} \bp_{n,k}\Big].
\]
Changing the summation according to the scheme
\begin{align*}
\sum_{i=0}^I d_i\sum_{k=1}^{m_i} u_k
= \sum_{i=0}^I \sum_{j=0}^I \sum_{k\in A_j} d_iu_k\ind_{\{j\le i\}}
= \sum_{j=0}^I\sum_{k\in A_j} \sum_{i=0}^Id_iu_k\ind_{\{j\le i\}}
= \sum_{j=0}^I\sum_{k\in A_j}u_k\sum_{i=j}^I d_i, 
\end{align*}
we obtain 
\begin{align*}
\P\left(\max_{1\le k\le n}\beta_k^{-1}|S_{n,k}|>\d\right)
&\le
c\d^{-r}\Big[\Big(\sum_{j=0}^I\sum_{k\in A_j}\ap_{n,k}\sum_{i=j}^I 2^{-2i/r}\Big)^{r/2}+  
\sum_{j=0}^I\sum_{k\in A_j}\bp_{n,k}\sum_{i=j}^I 2^{-i}\Big]\\ 
&\le 
2^{r/2}c\d^{-r}\Big[\Big(\sum_{j=0}^I\sum_{k\in A_j}2^{-2j/r}\ap_{n,k}\Big)^{r/2}
+\sum_{j=0}^I\sum_{k\in A_j}2^{-j}\bp_{n,k}\Big]\\
&\le 
2^{1+r/2}c\d^{-r}\Big[\Big(\sum_{j=0}^I\sum_{k\in A_j}\beta_{n,k}^{-2}\ap_{n,k}\Big)^{r/2}
+\sum_{j=0}^I\sum_{k\in A_j}\beta_{n,k}^{-r}\bp_{n,k}\Big]\\
&= 
2^{1+r/2}c\d^{-r}\Big[\Big(\sum_{k=1}^n\beta_{n,k}^{-2}\ap_{n,k}\Big)^{r/2}
+\sum_{k=1}^n\beta_{n,k}^{-r}\bp_{n,k}\Big]
\end{align*}
and the result is proved. 
\end{proof}

We will need also the following inequality for which we refer to
\cite[Lemma 2]{MRS:2012a}.

\begin{lemma}\label{lemma:maxq}
  Let $(\eta_j)_{j\ge 0}$ be a sequence of \iid\  random variables
  with $\E\eta_0 = 0$ and $\E|\eta_0|^q < \infty$ for some $q\ge
  2$. Suppose that $\p_n\to 1$ and $n(1-\p_n)\to \infty$ as
  $n\to\infty$.  Then there exists an integer $n_0(q)\ge 1$, depending
  on $q$ only, such that, for all $n\ge n_0(q)$, and $\lambda > 0$,
\begin{equation}\label{max:q}
\P\Big(\max_{1\le k\le n}\Big|\sum_{j=1}^k \p_n^{k-j}\eta_j\Big|>\lambda\Big)
\le 4C_q\eexp{q}\lambda^{-q}n^{q/2}\E|\eta_0|^q\big(n(1-\p_n)\big)^{1-q/2},
\end{equation}
where $C_q$ is the universal constant in the Rosenthal inequality of order $q$.
\end{lemma}

Finally the following estimates for truncated moments will be useful.

\begin{lemma}\label{lem-trunc-mom}
Let $X$ be a non negative random variable.
\begin{enumerate}[a)]
\item Assume that $X\in \Lpi$ for some $p>1$ and put
  $N_p(X):=\sup_{t>0}t^p\P(X>t)$. Then for every $u>0$,
\begin{equation}\label{250314a}
\E \big(X\ind_{\{X\ge u\}}\big) \le \frac{p}{p-1}N_p(X)u^{1-p}
\end{equation}
and for every $q>p$,
\begin{equation}\label{250314b}
\E \big(X\ind_{\{X\le u\}}\big)^q \le \frac{q}{q-p}N_p(X)u^{q-p}.
\end{equation}

\item Assume that $\P(X>t)$ is regularly varying with index $-p$ (this
  condition is satisfied in particular when $X\in\RV_p$). Let
  $b_n =\inf\{t>0: \P(X\le t)\ge 1-1/n\}$. Then for any $\delta>0$,
\begin{equation}\label{250314c}
\E \big(X\ind_{\{X\ge hb_n\}}\big)\le \frac{p(1+\delta)}{p-1} h^{1-p}b_nn^{-1},
\end{equation}
for $n$ large enough, uniformly in $h\in[1,\infty)$. For any $r>p$,
any $\delta>0$, 
\begin{equation}\label{250314d}
\E \big(X\ind_{\{X\le hb_n\}}\big)^r \le \frac{r(1+\delta)}{r-p}h^{r-p}b_n^r n^{-1},
\end{equation}
for $n$ large enough, uniformly in $h\in[1,\infty)$.
\end{enumerate}
\end{lemma}

\begin{proof}
To prove~\eqref{250314a} we observe that 
$\P\big(X\ind_{\{X\ge u\}} > t\big) = \P(X> \max(t,u))$,
for any $t>0$, whence
\begin{align*}
\E\big(X\ind_{\{X\ge u\}}\big) 
= 
\int_0^\infty \P\big(X\ind_{\{X\ge u\}} > t\big)\dd t
&= \int_0^u \P(X>u)\dd t + \int_u^\infty \P(X>t)\dd t\\
&\le
u\P(X>u) + \int_u^\infty \frac{N_p(X)}{t^p}\dd t \\
&\le
N_p(X)u^{1-p} + N_p(X)\frac{u^{1-p}}{p-1} = \frac{pN_p(X)}{p-1}u^{1-p}.
\end{align*}

To prove~\eqref{250314b}, we note that for $t$, $u>0$,
$\P\big(X\ind_{\{X\le u\}} > t\big) = \P(t<X\le u)$, 
whence
\begin{align*}
\E\big(X\ind_{\{X\le u\}}\big)^q
=
\int_0^\infty qt^{q-1} \P\big(X\ind_{\{X\le u\}} >t\big)\dd t
&=
\int_0^\infty qt^{q-1} \P(t<X\le u)\dd t\\
&=
\int_0^u qt^{q-1} \P(t<X\le u)\dd t\\
&\le
\int_0^u qt^{q-1}\P(X>t)\dd t\\
&\le
\int_0^u qt^{q-p-1}N_p(X)\dd t = \frac{qN_p(X)}{q-p}u^{q-p}.
\end{align*}

The proof of~\eqref{250314c} starts like the one of~\eqref{250314a}:
\begin{equation}\label{260314a}
\E \big(X\ind_{\{X\ge hb_n\}}\big) =  
hb_n \P(X>hb_n) + \int_{hb_n}^\infty \P(X>t)\dd t.
\end{equation}
Now, as $\P(X>t)$ is regularly varying with index $-p$, by
Prop. 1.5.10 in \cite{BGT:1987},
\[
\frac{hb_n\P(X>hb_n) }{\int_{hb_n}^\infty \P(X>t)\dd t}
\xrightarrow[\;n\to\infty\;]{} p-1,\quad\text{uniformly in $h\ge 1$,}
\]
which combined with~\eqref{260314a} gives
\begin{equation}\label{260314b}
\E \big(X\ind_{\{X\ge hb_n\}}\big)\sim \frac{p}{p-1}hb_n\P(X>hb_n),
\quad\text{uniformly in $h\ge 1$}.
\end{equation}
By th. 1.5.2 in \cite{BGT:1987},
\begin{equation}\label{260314c}
\P(X>hb_n)\sim h^{-p}\P(X>b_n),\quad\text{uniformly in $h\ge 1$}.
\end{equation}
From the definition of the quantile $b_n$, $\P(X>b_n)\le n^{-1}$. 
Combining this estimate with~\eqref{260314b} and~\eqref{260314c}
gives~\eqref{250314c}.

To prove~\eqref{250314d}, we begin by noting that like in the proof
of~\eqref{250314b},
\[
\E \big(X\ind_{\{X\le hb_n\}}\big)^r \le
\int_0^{hb_n}rt^{r-1}\P(X>t)\dd t.
\]
Next, by Th. 1.5.11 (i) in~\cite{BGT:1987},
\[
\frac{(hb_n)^r\P(X>hb_n) }{\int_0^{hb_n}t^{r-1} \P(X>t)\dd t}
\xrightarrow[\;n\to\infty\;]{} r-p,\quad\text{uniformly in $h\ge 1$,}
\]
from which we obtain~\eqref{250314d} in the same way as for~\eqref{250314c}. 
\end{proof}

\subsection{Proof of Lemma~\ref{Lem-181113A}, common part}
Since
$$
\sum_{j=k+1}^{k+\ell}y_{n,j-1}=\frac{1}{1-\p_n}
\Big[\sum_{j=k+1}^{k+\ell}\e_j+y_{n,k}-y_{n,k+\ell}\Big]
$$
and 
$$
y_{n,k+\ell}-y_{n,k}=\sum_{j=1}^{k+\ell}\p_n^{k+\ell-j}\e_j-\sum_{j=1}^{k}\p_n^{k-j}\e_j=
\sum_{j=k+1}^{k+\ell}\p_n^{k+\ell-j}\e_j-(1-\p_n^\ell)y_{n,k}
$$
for any $1\le \ell\le n$, $0\le k\le n-\ell$ and $y_{n,0}=0$, we deduce
\[
\abs{(1-\p_n)T_{\alpha, n}(y_{n,0},\dots,y_{n,n-1}) - T_{\alpha, n}(\e_1,\dots,\e_n)}
\le T_n^{(1)}+T_n^{(2)}+T_n^{(3)},
\]
where
\begin{align*}
T_n^{(1)}
&:=
\max_{1\le \ell<n}\ell^{-\alpha} \max_{1\le k\le  n-\ell}\Big|\sum_{j=k+1}^{k+\ell}\p_n^{k+\ell-j}\e_j\Big|,\\
T_n^{(2)}
&:=
\max_{1\le \ell<n}\ell^{-\alpha} \max_{1\le k\le  n-\ell}
\abs{(1-\p_n^\ell)y_{n,k}}\\
T_n^{(3)}
&=
\max_{1\le \ell<n}\ell^{-\alpha} (\ell/n)|y_{n,n}|=n^{-\alpha}|y_{n,n}|.
\end{align*}
We will show separately below that in the case a) of Lemma~\ref{Lem-181113A}
\begin{equation}\label{tn:123}
T_n^{(i)}=O_P(b_n),\quad i=1, 2, 3,
\end{equation}
while in case b)
\begin{equation}\label{tn:123b}
T_n^{(i)}=O_P(n^{1/2-\a}),\quad i=1, 2, 3.
\end{equation}
This will establish \eqref{181113a} and \eqref{181113A}. 

Before this splitting of the proof, we can already note
that~\eqref{181113b} and~\eqref{181113B} respectively follow from
\eqref{181113a} and \eqref{181113A}.  Indeed for (\ref{181113b}) we
observe that
\begin{equation}\label{020214a}
|T_{\alpha, n}(\wh{\e}_1, \dots, \wh{\e}_n)-T_{\alpha,n}(\e_1, \dots, \e_n)|\le 
|\wh{\p}_n-\p_n|T_{\alpha, n}(y_{n,0}, \dots, y_{n,n}).
\end{equation}
Since $T_{\alpha, n}(\e_1, \dots, \e_n)=O_P(b_n)$ by
Theorem~1.1. in~\cite{MR:2010}, (\ref{181113a}) gives
\[
(1-\p_n)T_{\alpha, n}(y_{n,0}, \dots, y_{n,n})=O_P(b_n).
\]  
By Giraitis and Philips~\cite[Th. 1]{Giraitis-Phillips:2004},
$n^{1/2}(1-\p_n^2)^{-1/2}(\wh{\p}_n-\p_n)$ is asymptoticaly normal
provided that $\E\e_1^2<\infty$, so
\[
\wh{\p}_n-\p_n=O_P\left(\frac{(1-\p_n)^{1/2}}{n^{1/2}}\right)
\]
We then deduce from \eqref{020214a} that
\begin{align*}
T_{\alpha, n}(\wh{\e}_1, \dots, \wh{\e}_n)-T_{\alpha,n}(\e_1, \dots, \e_n)
& = 
O_P\left(\frac{b_n\abs{\wh{\p}_n-\p_n}}{1-\p_n}\right)\\
& = 
O_P\left(\frac{b_n}{n^{1/2}(1-\p_n)^{1/2}}\right) =o_P(b_n).
\end{align*}
since $n(1-\p_n)$ tends to infinity. 

The deduction of \eqref{181113B} from \eqref{181113A} is essentially
the same. We just have to replace $b_n$ by $n^{1/2-\a}$ and note that
the estimate $T_{\alpha, n}(\e_1, \dots, \e_n)=O_P(n^{1/2-\a})$ is a
by-product of the convergence in distribution of
$n^{-1/2+\a}T_{\alpha,n}(\e_1,\ldots,\e_n)$ already established in the proof
of Theorem~\ref{Th301013a}.

\subsection{Proof of Lemma~\ref{Lem-181113A},  estimate (\ref{181113a}) }

\noindent\textsl{Estimate for $T_n^{(1)}$.}
For $h>0$ set
$$
P^{(1)}_{n, h}:=\P(T_n^{(1)}>2h b_n).
$$
To estimate this probability we define the truncated random variables:
$$
\e'_j=\e_j\ind_{\{|\e_j|> h b_n\}}, \quad
\e''_j=\e_j\ind_{\{|\e_j|\le h b_n\}}-\E \e_j\ind_{\{|\e_j|\le h b_n\}},
$$
for $j\ge 1$. 
Then 
$$
P^{(1)}_{n, h}\le P^{(1,1)}_{n, h}+P^{(1,2)}_{n, h},
$$
where
$$
P^{(1,1)}_{n, h}:=\P\left(\max_{1\le j\le n}|\e_j|>h b_n\right),
\quad 
P^{(1,2)}_{n, h}:=\P(T_n^{(1,2)}>2h b_n)
$$
with 
$$
T_n^{(1,2)}:=
\max_{1\le \ell<n}\ell^{-\alpha} \max_{1\le k\le  n-\ell}
\Big|\sum_{j=k+1}^{k+\ell}\p_n^{k+\ell-j}\e_j\bm{1}_{\{|\e_j|\le hb_n\}}\Big|.
$$
From extreme value theory we know (see, for example \cite{EKM:1997},
Theorem 3.3.7) that
$$
\P\left(\max_{1\le j\le n}|\e_j|>h b_n\right)\to 1-\exp\{-h^{-p}\}
$$
as $n\to \infty$. Choosing $h$ big enough we make probability
$P^{(1,1)}$ arbitrary small, in other words
\begin{equation}\label{eq:p11}
\lim_{h\to\infty}\limsup_{n\to\infty}P_{n, h}^{(1,1)}=0.
\end{equation}

Next we estimate $P_{n, h}^{(1,2)}$. First we need to center each
$\e_i\bm{1}_{\{|\e_i|\le hb_n\}}$, $i=1, \dots, n$. Observing that
 $\E\e_i=0$, $\E\e_i\bm{1}_{\{|\e_i|\le  hb_n\}}=\E\e_i\bm{1}_{\{|\e_i|> hb_n\}}$,
 we have
$$ 
\max_{1\le \ell\le n}\ell^{-\alpha}
\Big|\sum_{j=k+1}^{k+\ell}\p_n^{k+\ell-j}\E[\e_j\bm{1}_{\{|\e_j|\le hb_n\}}]\Big|
\le 
n^{1-\alpha} \E|\e_1|\bm{1}_{\{|\e_1|\ge h b_n\}}.
$$
By \eqref{250314c} in Lemma~\ref{lem-trunc-mom},
\begin{equation}\label{260314f}
\E|\e_1|\ind_{\{|\e_1|\ge hb_n\}}\le 2p(p-1)^{-1}n^{-1}b_nh^{1-p}
\end{equation}
for  $n$ large enough uniformly in $h\ge 1$, we obtain
\begin{align*}
\max_{1\le \ell\le n}\ell^{-\alpha}
\Big|\sum_{j=k+1}^{k+\ell}\p_n^{k+\ell-j}\E[\e_j\bm{1}_{\{|\e_j|\le hb_n\}}]\Big|
&\le \frac{2p}{(p-1)n^\alpha h^p}hb_n  \le hb_n 
\end{align*}
for  $n$ large enough \emph{uniformly} in $h\ge 1$. 
It follows that for $n$ large enough and $h\ge 1$, 
$$
P_{n, h}^{(1,2)}\le \P(T_n^{(1,2,1)}>h b_n)
$$
with 
$$
T_n^{(1,2,1)}:=\max_{1\le \ell<n}\ell^{-\alpha} \max_{1\le k\le  n-\ell}\Big|\sum_{j=k+1}^{k+\ell}\p_n^{k+\ell-j}\e''_j\Big|.
$$
Since
$$
T_n^{(1,2,1)}=\max_{1\le \ell<n}\ell^{-\alpha} \max_{1\le k\le  n-\ell}\Big|\sum_{j=1}^{\ell}\p_n^{\ell-j}\e''_{k+j}\Big|\le
\max_{1\le k\le n}\max_{1\le \ell\le n}\ell^{-\alpha}\Big|\sum_{j=1}^{\ell}\p_n^{\ell-j}\e''_{k+j}\Big|
$$
we have 
\begin{align*}
\P(T_n^{(1,2,1)}\ge h b_n)
&\le 
\sum_{k=1}^n \P\left(\max_{1\le \ell\le n}\ell^{-\alpha}
\Big|\sum_{j=1}^{\ell}\p_n^{\ell-j}\e''_{k+j}\Big|\ge h b_n\right)\\
&=
n\P\left(\max_{1\le \ell\le n}\ell^{-\alpha}\Big|\sum_{j=1}^{\ell}\p_n^{\ell-j}\e''_{j}\Big|\ge h b_n\right)
\end{align*}
due to stationarity.  Choose $r>p$. Using
successively Markov's, Doob's and Rosenthal's inequalities, we obtain
for each $\d>0$ 
\[
\P\Big(\max_{1\le \ell\le n}
\Big|\sum_{j=1}^\ell\p_n^{-j}\e''_j\Big|>\d\Big)\le
c\d^{-r}\Big[\Big(\sum_{j=1}^{\chgt{n}}
\p_n^{-2j}\E(\e''_1)^2\Big)^{r/2}+\sum_{j=1}^{n}
\p_n^{-rj}\E|\e''_1|^r \Big]
\]
with a constant $c>0$ depending on $r$ only. 
Using \eqref{250314d} in Lemma~\ref{lem-trunc-mom} together with the inequality $(\abs{a}+\abs{b})^r\le 2^{r-1}(\abs{a}^r+\abs{b}^r)$, we obtain
\begin{equation}\label{260314g}
\E\abs{\e_1''}^r \le 2^r\E\big(\abs{\e_1}\ind_{\{|\e_1|\le h b_n\}}\big)^r \le
\frac{r2^{r+1}}{r-p}h^{r-p}b_n^r n^{-1},
\end{equation} 
for $n$ large enough, uniformly in $h\in[1,\infty)$.
Hence, there is a constant $c>0$ 
depending on $r$ and $p$ only, such that for $n$ large enough and $h\ge 1$,
\begin{align}
 \P\Big(\max_{\chgt{1\le \ell\le n}}\Big|\sum_{j=1}^\ell \p_n^{-j}\e''_j\Big|>\d\Big)
&\le 
c\d^{-r}\Big[\Big(\sum_{k=1}^n \ap_{n,k}\Big)^{r/2}+\sum_{k=1}^n \bp_{n,k}\Big],
\label{sums:1}
\end{align}
where $\ap_{n,k}=\sigma^2\p_n^{-2k}$ and 
$\bp_{n,k}=\p_n^{-rk}h^{r-p}n^{-1}b^r_n$, 
$k=1, \dots, n$.  By lemma \ref{H_R} we deduce
\begin{align*}
\P\left(\max_{1\le \ell\le n}\ell^{-\alpha}
\Big|\sum_{j=1}^{\ell}\p_n^{\ell-j}\e''_{j}\Big|\ge h b_n\right)
&\le 
c h^{-r}b_n^{-r}\Big[\Big(\sum_{k=1}^n k^{-2\alpha}\p_n^{2k} \ap_{n,k}\Big)^{r/2}+
\sum_{k=1}^n \p_n^{rk} k^{-r\alpha}\bp_{n,k}\Big]\\
&=
c h^{-r}b_n^{-r}\Big[\chgt{\sigma^r}\Big(\sum_{k=1}^n k^{-2\alpha}\Big)^{r/2}+
\sum_{k=1}^n k^{-r\alpha}h^{r-p}n^{-1}b_n^r\Big].
\end{align*}
Since $\sum_{k=1}^n k^{-d\alpha}\chgt{=O(\max\{n^{1-d\alpha}, 1\})}$
for $d>0$, choosing $r > 1/\alpha$ we have with some positive
  constant $C$,
\[
\P\left(\max_{1\le \ell\le n}\ell^{-\alpha}
\Big|\sum_{j=1}^{\ell}\p_n^{\ell-j}\e''_{j}\Big|\ge h b_n\right)
\le 
Ch^{-r}b_n^{-r}\big(\max\{n^{r(1-2\alpha)/2}, 1\}+h^{r-p}n^{-1}b_n^r\big).
\]
Hence,
\[
\P(T_n^{(1,2,1)}>hb_n)\le \chgt{C}
\big( h^{-r}b_n^{-r}\max\{n, n^{1+r(1-2\alpha)/2}\}+h^{-p}\big).
\]
By choosing $r>(1/p-1/2+\alpha)^{-1}$ we see that
$$
\lim_{h\to\infty}\limsup_{n\to\infty}\P(T_n^{(1,2,1)}\ge h b_n)=0.
$$
The proof of (\ref{tn:123}) for $i=1$  is now complete.

\medskip
\noindent\textsl{Estimate for $T_n^{(2)}$.}
First we note that
\[
T_n^{(2)}=
\max_{1\le \ell<n}\ell^{-\alpha}(1-\p_n^\ell)\max_{1\le k\le  n-\ell}\abs{y_{n,k}}
\le
\left(\max_{1\le \ell<n}\ell^{-\alpha}(1-\p_n^\ell)\right)\max_{1\le j\le n}\abs{y_{n,j}}.
\]
Then we observe that
\begin{align*}
\max_{1\le \ell<n}\ell^{-\alpha}(1-\p_n^\ell) \le \sup_{t\ge 1}t^{-\a}(1-\p_n^t)
=\abs{\ln \p_n}^\a\sup_{u\ge\abs{\ln \p_n}}u^{-\a}(1-\eexp{-u})
\le \abs{\ln \p_n}^\a.
\end{align*}
Since $\abs{\ln \p_n}=\ln(1/\p_n)\le \p_n^{-1}-1$, we obtain
\[
\max_{1\le \ell<n}\ell^{-\alpha}(1-\p_n^\ell) \le \frac{(1-\p_n)^\a}{\p_n^\a}=
O\big((1-\p_n)^\a\big).
\]
Next we  prove that 
\begin{equation}\label{last:1}
\max_{1\le k\le n}|y_{n,k}|=O_P(b_n(1-\p_n)^{-\alpha}).
\end{equation}
To this aim we use similar techniques as above. Set for $h>0$
$$
P_n^{(2)}=\P(\max_{1\le k\le n}|y_{n,k}|>2hb_n(1-\p_n)^{-\alpha}).
$$
Then $P_n^{(2)}\le P_n^{(2,1)}+P_n^{(2,2)}$, where
\begin{align*}
 P_n^{(2,1)}&=\P(\max_{1\le k\le n}|\e_k|>hb_n),\\
 P_n^{(2,2)}&=\P\Big(\max_{1\le k\le n}\Big|\sum_{j=1}^k \p_n^{k-j}\e_j\bm{1}_{\{|\e_j|\le hb_n\}}\Big|\ge 2hb_n(1-\p_n)^{-\alpha}\Big).
\end{align*}
Since $P_n^{(2,1)}=P_n^{(1,1)}$ we have from (\ref{eq:p11})
$$
\lim_{h\to\infty}\limsup_{n\to\infty}P_n^{(2,1)}=0.
$$
Since
$$
\max_{1\le k\le n}\Big|\sum_{j=1}^k \p_n^{k-j}\E\e_j\ind_{\{|\e_j|\le hb_n\}}\Big|
\le 
n\E|\e_1|\ind_{\{|\e_1|\ge hb_n\}}\le 
\frac{2p}{p-1} b_n h^{1-p},
$$
using again~\eqref{260314f}, we deduce
$$
P_n^{(2,2)}\le 
\P\Big(\max_{1\le k\le n}\Big|\sum_{j=1}^k \p_n^{k-j}\e''_j\Big|\ge hb_n(1-\p_n)^{-\alpha}\Big)
$$
for  $n$ large enough
and $h\ge (2p)^{1/p}(p-1)^{-1/p}$.
Now  we apply lemma \ref{lemma:maxq} and obtain
$$
P_n^{(2,2)}\le ch^{-q}b_n^{-q}(1-\p_n)^{q\alpha}\E|\e_1''|^qn^{q/2}(n(1-\p_n))^{1-q/2},
$$
with a constant depending only on $q$.
Using~\eqref{260314g} to bound $\E|\e_1''|^q$ we deduce
$$
P_n^{(2,2)}\le ch^{-p}(1-\p_n)^{q\alpha+1-q/2}.
$$
where $c$ depends on $p$ and $q$ only.

If $\alpha\ge 1/2$, then $1+q\alpha-q/2>0$ and so
\begin{equation}\label{last:2}
\lim_{n\to\infty}P_n^{(2,2)}=0.
\end{equation}
If $\alpha<1/2$ then $1+q\alpha-q/2>0$ provided $q<1/(1/2-\alpha)$.
So we have to choose $p<q<1/(1/2-\alpha)$ in the case $\alpha
<1/2$. This is possible since $\alpha>1/2-1/p.$ So in any case
(\ref{last:2}) is valid and (\ref{tn:123}) with $i=2$ follows.

\medskip
\noindent\textsl{Estimate for $T_n^{(3)}$.}
We have $\E y_n^2=\sum_{k=1}^n \p_n^{2(n-k)}\sigma^2\le
n\sigma^2$. Hence, $|y_n|=O_P(n^{1/2})$ and we obtain
$T_n^{(3)}=n^{-\alpha}O_P(n^{1/2})=o_P(b_n)$ recalling that
$b_n=n^{1/p}v(n)$ with a slowly varying  function $v(n)$ and observing that
\[
n^{-\alpha}n^{1/2}b_n^{-1}= n^{-\alpha-(1/p)+(1/2)}v(n)^{-1}\to 0
\]
since $\alpha>1/2-1/p$.

The proof of (\ref{181113a}) is now complete.

\subsection{Proof of Lemma~\ref{Lem-181113A},  estimate (\ref{181113A}) }
Next we consider the case (b) and prove (\ref{tn:123b}) for $i=1, 2, 3.$

\medskip
\noindent\textsl{Estimate for $T_n^{(1)}$.}
 Set $p_\alpha=1/(1/2-\alpha)$.
For $h>0$ set
$$
P^{(1)}_{n, h}:=\P(T_n^{(1)}>2hn^{1/p_\alpha}).
$$
To estimate this probability we define the truncated random variables:
$$
\e'_j=\e_j\bm{1}\{|\e_j|> h n^{1/p_\alpha}\}, \quad
\e''_j=\e_j\bm{1}\{|\e_j|\le h n^{1/p_\alpha}\}
-\E \e_j\bm{1}\{|\e_j|\le h n^{1/p_\alpha}\},
$$
for $j\ge 1$. 
Then 
$$
P^{(1)}_{n, h}\le P^{(1,1)}_{n, h}+P^{(1,2)}_{n, h},
$$
where
$$
P^{(1,1)}_{n, h}:=\P\left(\max_{1\le j\le n}|\e_j|>h n^{1/p_\alpha}\right),
\quad 
P^{(1,2)}_{n, h}:=\P\left(T_n^{(1,2)}>2h n^{1/p_\alpha}\right)
$$
with 
\[
T_n^{(1,2)}:=
\max_{1\le \ell<n}\ell^{-\alpha} \max_{1\le k\le  n-\ell}
\Big|\sum_{j=k+1}^{k+\ell}\p_n^{k+\ell-j}\e_j\bm{1}_{\{|\e_j|\le hn^{1/p_\alpha}\}}\Big|.
\]
Since
$$
\P\left(\max_{1\le j\le n}|\e_j|>h n^{1/2-\alpha}\right)\le 
n\P(|\e_1|>h n^{1/2-\alpha}),
$$
$P^{(1,1)}_{n, h}$ tends to $0$ as $n\to \infty$ due to the condition
$\e_i\in \Lpio$ and $0\le\alpha\le 1/2-1/p$ when $p>2$ or to the condition 
$\e_i\in\mathcal{L}_2$ when $p=2$ (then $\a=\a_2=0$).

Next we estimate $P_{n, h}^{(1,2)}$. First we need to center each
$\e_i\bm{1}_{\{|\e_i|\le hn^{1/p_\alpha}\}}$, $i=1, \dots,
n$. 
From~\eqref{250314a} we get that
\[
\E\abs{\e_1}\ind_{\{\abs{\e_1}\ge h n^{1/p_\a}\}} \le c\big(hn^{1/p_\a}\big)^{1-p},
\]
where $c=p(p-1)^{-1}\sup_{t>0}\P(\abs{\e_1}>t)$.
As
$$ 
\max_{1\le \ell\le n}\ell^{-\alpha}
\Big|\sum_{j=k+1}^{k+\ell}\p_n^{k+\ell-j}\E[\e_j\bm{1}_{\{|\e_j|\le hn^{1/p_\alpha}\}}]\Big|
\le 
n^{1-\alpha} \E|\e_1|\ind_{\{|\e_j|\ge h n^{1/p_\alpha}\}}
$$
it follows that for every $h\ge h_0:=\max(1,c)$,
\[
n^{1-\a}\E\abs{\e_1}\ind_{\{\abs{\e_1}\ge hn^{1/p_\a}\}}
\le ch^{1-p}n^{1-\a+1/p_\a-p/p_\a}
\le c h^{-1}n^{1-p/p_\a}n^{1/p_\a}n^{-\a}
\le n^{1/p_\a},
\]
recalling that $p\ge p_\a$ since $\a\le \a_p$.
Hence for every $n\ge 1$ and every $h\ge h_0$,
\[
\P\left(\max_{1\le \ell\le n}\ell^{-\alpha}
\Big|\sum_{j=k+1}^{k+\ell}\p_n^{k+\ell-j}\E[\e_j\bm{1}_{\{|\e_j|\le hn^{1/p_\alpha}\}}]\Big|
> hn^{1/p_\a} \right) = 0
\]
and we deduce
$$
P_{n, h}^{(1,2)}\le \P(T_n^{(1,2,1)}>h n^{1/p_\alpha})
$$
with 
$$
T_n^{(1,2,1)}:=
\max_{1\le \ell<n}\ell^{-\alpha} \max_{1\le k\le  n-\ell}\Big|\sum_{j=k+1}^{k+\ell}\p_n^{k+\ell-j}\e''_j\Big|.
$$
Since
$$
T_n^{(1,2,1)}=\max_{1\le \ell<n}\ell^{-\alpha} \max_{1\le k\le  n-\ell}\Big|\sum_{j=1}^{\ell}\p_n^{\ell-j}\e''_{k+j}\Big|\le
\max_{1\le k\le n}\max_{1\le \ell\le n}\ell^{-\alpha}\Big|\sum_{j=1}^{\ell}\p_n^{\ell-j}\e''_{k+j}\Big|
$$
we have 
\begin{align*}
\P(T_n^{(1,2,1)}\ge h n^{1/p_\a})
&\le 
\sum_{k=1}^n \P\left(\max_{1\le \ell\le n}\ell^{-\alpha}
\Big|\sum_{j=1}^{\ell}\p_n^{\ell-j}\e''_{k+j}\Big|\ge h n^{1/p_\alpha}\right)\\
&=
n\P\left(\max_{1\le \ell\le n}\ell^{-\alpha}\Big|\sum_{j=1}^{\ell}\p_n^{\ell-j}\e''_{j}\Big|\ge h n^{1/p_\alpha}\right).
\end{align*}
Let $K>1$  and let $MK=n$ where $M$ and $K$ (not necessarily integers) depend on $n$ in a way which will be precised later. Splitting the set
$$
\{1, \dots, n\}=\bigcup_{m=1}^M \left(\N\cap \big((m-1)K, mK\big]\right),
$$
we have
\begin{align*}
\max_{1\le \ell\le n}\ell^{-\alpha}\Big|\sum_{j=1}^{\ell}\p_n^{\ell-j}\e''_{j}\Big|&\le 
\max_{1\le m\le  M}\max_{(m-1)K<\ell\le mK}\ell^{-\alpha}\Big|\sum_{j=1}^{\ell}\p_n^{\ell-j}\e''_{j}\Big|\\
&\le 
\max_{1\le m\le  M}\max_{(m-1)K<\ell\le mK}[(m-1)K+1]^{-\alpha}\p_n^{(m-1)K+1}\Big|\sum_{j=1}^{\ell}\p_n^{-j}\e''_{j}\Big|
\end{align*}
This leads to
$$
\P(T_n^{(1,2,1)}\ge h n^{1/p_\a})
\le 
n\sum_{m=1}^M \P\left(\max_{(m-1)K< \ell\le mK}\Big|\sum_{j=1}^{\ell}\p_n^{\ell-j}\e''_{j}\Big|\ge h n^{1/p_\a}\lambda_{m,K}\right)
$$
where $\lambda_{m,K}=[(m-1)K+1]^{\alpha}\p_n^{-(m-1)K-1}$.  Let
$q>p\ge p_\a$, whose choice will be precised
later. 
As $p_\a\le p$, $\e_1\in\mathcal{L}_{p_\a}$, so
by~\eqref{250314b},
\begin{equation}\label{270314a}
\E\abs{\e''_1}^q\le ch^{q-p_\a}n^{(q-p_\alpha)/p_\alpha},
\end{equation}
where the constant $c$ depends on $q$, $\a$ and the distribution of
$\e_1$ only.  
Using successively Markov's, Doob's and Rosenthal's
inequalities, we obtain
\begin{align*}
\P(T_n^{(1,2,1)}\ge h n^{1/p_\a})
&\le 
n\sum_{m=1}^M (hn^{1/p_\a}a_{m,K})^{-q}\E\Big|\sum_{j=1}^{mK}\p_n^{-j}\e''_{j}\Big|^q\\
&\le 
c_q n\sum_{m=1}^M (hn^{1/p_\a}a_{m,K})^{-q}
\Big[\Big(\sum_{j=1}^{mK}\p_n^{-2j}\Big)^{q/2}+
\sum_{j=1}^{mK}\p_n^{-qj}\E|\e_1''|^q\Big]\\
&\le 
C_qn(hn^{1/p_\alpha})^{-q}K^{-q\alpha}\sum_{m=1}^Mm^{-q\alpha}\p_n^{-qK}
\left[\frac{1}{(1-\p_n^2)^{q/2}}+\frac{\E|\e''_1|^q}{(1-\p_n^q)}\right],
\intertext{
using~\eqref{270314a} 
and recalling the restriction $h\ge h_0\ge  1$, we continue by}
&\le  
C'_q n^{1-q/p_\a}h^{-p_\a}                       
K^{-q\alpha}\sum_{m=1}^Mm^{-q\alpha}\p_n^{-qK}
\left[\frac{1}{(1-\p_n^2)^{q/2}}+\frac{n^{q/p_\a -1}}{(1-\p_n^q)}\right]\\
&\le  
C_{q,\a} n^{1-q/p_\a}h^{-p_\a}
\p_n^{-qK}K^{-q\alpha}
\left[\frac{1}{(1-\p_n^2)^{q/2}}+\frac{n^{q/p_\a -1}}{(1-\p_n^q)}\right],
\end{align*}
since $q>1/\alpha$. Now choosing $K\sim (1-\p_n)^{-1}$ and observing that 
$\p_n^{-qK}\sim \eexp{q}$, we finally have
\begin{align*}
P(T_n^{(1,2,1)}\ge h n^{1/p_\a})
&\le 
C'_{q,\a} n^{1-q/p_\a}h^{-p_\a} (1-\p_n)^{q\alpha}
\left[\frac{1}{(1-\p_n^2)^{q/2}}+\frac{n^{q/p_\a -1}}{(1-\p_n^q)}\right]\\
&\le 
C'_{q,\a} h^{-p_\a}\Big[\frac{n}{(n(1-\p_n))^{q/p_\alpha}}+(1-\p_n^q)^{q\alpha-1}\Big].
\end{align*}
Now, the hypothesis~\eqref{160314a} enables us to choose $q$ in such a
way that $n (n(1-\p_n))^{-q/p_\alpha}$ remains bounded, namely
$q>\max(p,p_\a/\delta)$, so we obtain
$T_n^{(1,2,1)}=O_P(n^{1/p_\alpha})$.

\medskip
\noindent\textsl{Estimate for $T_n^{(2)}$.}
As already seen in the proof of inequality~\eqref{181113a},
\[
\max_{1\le \ell<n}\ell^{-\alpha}(1-\p_n^\ell)=O\big((1-\p_n)^\a\big),
\]
so it remains only to check  that 
\begin{equation}\label{last:3}
\max_{1\le k\le n}|y_{n,k}|=O_P(n^{1/p_\alpha}(1-\p_n)^{-\alpha}).
\end{equation}
But this is known from~\cite[Lemma 1]{MRS:2012a}, with $o_P$ instead
of $O_P$.  

\medskip
\noindent\textsl{Estimate for $T_n^{(3)}$.}
Finally for $T_n^{(3)}$, we have $\E y_n^2=\sum_{k=1}^n
\p_n^{2(n-k)}\sigma^2\le n\sigma^2$. Hence, $y_{n,n}=O_P(n^{1/2})$ and
we obtain $T_n^{(3)}=n^{-\alpha}O_P(n^{1/2})$. 

\color{black}

\section{Proofs of consistency lemmas}

\subsection{Proof of Lemma \ref{L281213}}
Recalling~\eqref{eq:1}, \eqref{050213d} and \eqref{091113a}, we note that
\begin{align}
 \wh{\p}_n - \p_n  
&=
\dfrac{\sum_{k=1}^n y_{n,k-1}(y_{n,k} - \p_n y_{n,k-1})}{\sum_{k=1}^n y_{n,k-1}^2}
\notag\\
&=
\dfrac{\sum_{k=1}^n y_{n,k-1}(\e_k+a_{n,k})}{\sum_{k=1}^n y_{n,k-1}^2}\notag\\
&=
\dfrac{\sum_{k=1}^n z_{n,k-1}\e_k  + \sum_{k=1}^n \tau_{n,k-1}\e_k + 
\sum_{k=1}^n y_{n,k-1}a_{n,k}}{\sum_{k=1}^n y_{n,k-1}^2}.
\label{311213A}
\end{align}
To obtain an upper bound for $\abs{ \wh{\p}_n - \p_n}$, we treat
separately the three sums in the above numerator.  

First, since $k^\ast\ge \lambda n$,
\begin{align}
\label{311213B}
\sum_{k=1}^n y_{n,k-1}^2 \ge \sum_{k=1}^{k^\ast} y_{n,k-1}^2 = 
\sum_{k=1}^{k^\ast} z_{n,k-1}^2 \ge \sum_{k=1}^{[\lambda n]} z_{n,k-1}^2,
\end{align}
whence
\begin{align}
\dfrac{\abs{\sum_{k=1}^n z_{n,k-1}\e_k}}{\sum_{k=1}^n y_{n,k-1}^2}
\le
\dfrac{\abs{\sum_{k=1}^n z_{n,k-1}\e_k}}{\sum_{k=1}^{[\lambda n]} z_{n,k-1}^2}
&=
\abs{\wt{\p}_n - \p_n}\frac{\sum_{k=1}^{n} z_{n,k-1}^2}{\sum_{k=1}^{[\lambda n]} z_{n,k-1}^2},
\label{311213a}
\end{align}
where $\wt{\p}_n=(\sum_{k=1}^n z_{n,k}z_{n,k-1})/(\sum_{k=1}^n
z_{n,k-1}^2)$ is the least squares estimator of $\p_n$ associated to
the process $(z_{n,k})$. As already observed, when $(y_{n,k})$
satisfies $H_A$, $(z_{n,k})$ satisfies $H_0$ and is then a nearly 
nonstationary AR(1) process. By~\cite[Th. 1]{Giraitis-Phillips:2004},
\begin{align*}
   \wt{\p}_n - \p_n = O_P\left(n^{-1/2}(1-\p_n)^{1/2}\right).
\end{align*}
Rewriting this estimate as
$O_P\left((1-\p_n)\left(n(1-\p_n)\right)^{-1/2}\right)$ and recalling
that $n(1-\p_n)$ tends to infinity, we get
\begin{align}
    \label{311213b}
\wt{\p}_n - \p_n = o_P(1-\p_n).
\end{align}
Moreover by~\cite[Lem. 2]{Giraitis-Phillips:2004},
\begin{align}
    \label{220413a}
    \frac{1-\p_n^2}{n} \sum_{k=1}^n z_{n,k-1}^2 \pcon \s^2.
\end{align}
Recalling that $\g_n=n(1-\p_n)$ is assumed to be non decreasing in $n$
or regularly varying, it is easily  deduced from this weak law of large
numbers that
\begin{align}
    \label{311213c}    
\dfrac{\sum_{k=1}^{n} z_{n,k-1}^2}{\sum_{k=1}^{[n\l]} z_{n,k-1}^2} = O_P(1).
\end{align}
Going back to~\eqref{311213a} with the estimates \eqref{311213b} and
\eqref{311213c}, we obtain
\begin{align}
    \label{311213d}
\dfrac{\abs{\sum_{k=1}^n z_{n,k-1}\e_k}}{\sum_{k=1}^n y_{n,k-1}^2}
= o_P(1-\p_n).
\end{align}

For the second sum in the numerator in~\eqref{311213A}, a simple
variance computation provides
\begin{align}
\label{311213e}
  \sum_{k=1}^n \tau_{n,k-1} \e_k = O_P\left(\left(\sum_{k=1}^n
      \tau_{n,k-1}^2 \right)^{1/2} \right).
\end{align}
Then
\begin{align*}
\sum_{k=1}^n \tau_{n,k-1}^2 &= \sum_{k=1}^n \left(\sum_{j=1}^{k-1} \p_n^{k-1-j} a_{n,j}\right)^2 = a_n^2 \sum_{k=1}^n \left(\sum_{j=1}^{k-1} \p_n^{k-1-j} \ind_{\I_n^*}(j)\right)^2 
\end{align*}
and
\begin{align*}
\sum_{k=1}^n \left(\sum_{j=1}^{k-1} \p_n^{k-1-j} \ind_{\I_n^*}(j)\right)^2
&=
\sum_{k=k^*+1}^{m^*} \left(\sum_{j=k^*+1}^{k-1} \p_n^{k-1-j} \right)^2 
+ \sum_{k=m^*+1}^{n} \left(\sum_{j=k^*+1}^{m^*} \p_n^{k-1-j} \right)^2\\
&=
\sum_{k=k^*+1}^{m^*}\left(\frac{1-\p_n^{k-k^\ast-1}}{1-\p_n}\right)^2
+\sum_{k=m^*+1}^{n}\p_n^{2(k-m^\ast-1)}\left(\sum_{i=0}^{\ell^\ast-1}\p_n^i\right)^2\\
&\le
\frac{l^*}{(1-\p_n)^2} + \frac{l^{*2}}{1-\p_n^2}.
\end{align*}
Since $a_n^2\ell^\ast=o(n(1-\p_n))$, it follows that
\begin{align*}
\sum_{k=1}^n \tau_{n,k-1}^2 = o\left(n(1-\p_n) {-1}\right) + o(n\ell^\ast).
\end{align*}
Accounting~\eqref{311213e}, this gives
\begin{align}
\label{010114a}
\sum_{k=1}^n \tau_{n,k-1} \e_k = o_P\left(n^{1/2}(1-\p_n)^{-1/2}\right) 
+ o_P\left(\sqrt{n\ell^\ast}\,\right).
\end{align}
Next, by~\eqref{311213B} and~\eqref{220413a} and recalling that
$n(1-\p_n)$ is non decreasing or regularly varying in $n$, we see that
\begin{align}
\label{010114b}
\dfrac{1}{\sum_{k=1}^n y_{n,k-1}^2}=
O_P\left(\frac{1-\p_{[n\lambda]}}{[n\lambda]}\right)
=O_P\left(\frac{[n\lambda](1-\p_{[n\lambda]})}{[n\lambda]^2}\right)
=O_P\left(\frac{1-\p_n}{n}\right).
\end{align}
Finally, combining~\eqref{010114a} and~\eqref{010114b}, we obtain
\begin{align}
\dfrac{\abs{\sum_{k=1}^n \tau_{n,k-1}\e_k}}{\sum_{k=1}^n y_{n,k-1}^2}
&= 
o_P\left(\sqrt{\frac{1-\p_n}{n}}\right) 
+ o_P\left((1-\p_n)\sqrt{\frac{\ell^\ast}{n}}\right)\notag\\
&=
o_P\left(\frac{1-\p_n}{\sqrt{n(1-\p_n)}}\right) + o_P(1-\p_n)
=o_P(1-\p_n),
\label{010114c}
\end{align}
since $\ell^\ast < n$ and $n(1-\p_n)$ tends to infinity.

To deal with the contribution of $\sum_{k=1}^n y_{n,k-1}a_{n,k}$, we
note first that
\begin{align*}
\dfrac{\abs{\sum_{k=1}^n y_{n,k-1}a_{n,k}}}{\sum_{k=1}^n y_{n,k-1}^2}
&\le
\dfrac{\left(\sum_{k=1}^n y_{n,k-1}^2\right)^{1/2}
\left(\sum_{k=1}^n a_{n,k}^2\right)^{1/2}}{\sum_{k=1}^n y_{n,k-1}^2}\\
&=
\frac{\sqrt{\ell^\ast}\abs{a_n}}{\left(\sum_{k=1}^n y_{n,k-1}^2\right)^{1/2}}\\
&=
O_P\left(\sqrt{\ell^\ast}\abs{a_n}\sqrt{\frac{1-\p_n}{n}}\right),
\end{align*}
using~\eqref{010114b}. Due to~\eqref{281213A}, this gives
\begin{align}
\label{010114d}
\dfrac{\abs{\sum_{k=1}^n y_{n,k-1}a_{n,k}}}{\sum_{k=1}^n y_{n,k-1}^2} = o_P(1-\p_n).
\end{align}
Going back to the decomposition~\eqref{311213A} with the
estimates~\eqref{311213d}, \eqref{010114c} and~\eqref{010114d}, we
conclude that  $\abs{\wh{\p}_n - \p_n} = o_P(1-\p_n)$.

\subsection{Proof of Lemma~\ref{L170413d}}
  We use
  \begin{align}
    \label{170413b}
    \sum_{j=1}^{n} \tau_{n,k-1} = \frac{a_n}{1-\p_n} \left(\ell^* -
      \p_n^{n-m^*} \frac{1-\p_n^{\ell^*}}{1-\p_n}\right).
  \end{align}
  To prove \eqref{proo170413A} we have to consider all the possible
  configurations of the sets $\{k+1, \ldots , k + \ell\}$ and
  $\brc{k^*+1, \ldots , k^* + \ell^*}$. There are six configurations $
  I_1, \ldots , I_6 $. Denote for $v = 1, \ldots , 6$
  \begin{align*}
    T_{\a,n}^{(v)} = \max_{k,\ell \in I_v} \ell^{-\a}
    \abs{\sum_{j=k+1}^{k+\ell} \tau_{n,j-1} - \frac{\ell}{n}
      \sum_{j=1}^n \tau_{n,j-1}}.
  \end{align*}

  First consider configuration $I_1 := \{k,\ell : [k^*+1,m^*] \subset
  [k+1, k+\ell]\}$. We easily obtain
  \begin{align*}
    \begin{split}
      \sum_{j=k+1}^{k+\ell} \tau_{n,j-1} &= a_n \left[ \sum_{j=k^*+1}^{m^*} \sum_{i=0}^{j-k^*-2} \p_n^i + \p_n^{-1} \sum_{j=m^*+1}^{k +\ell} \p_n^j \sum_{i=k^*+1}^{m^*} \p_n^{-i} \right] \\
      &= \frac{a_n}{1-\p_n} \left[\ell^* - \p_n^{k+\ell -m*}
        \frac{1-\p_n^{\ell*}}{1-\p_n} \right].
    \end{split}
  \end{align*}
  Together with \eqref{170413b} we find
  \begin{align*}
    T_{\a,n}^{(1)} &= \frac{\abs{a_n}}{1-\p_n}  \max_{k,\ell \in I_1} \ell^{-\a} \abs{\ell^*(1-\ell/n) - \frac{1-\p_n^{\ell*}}{1-\p_n} (\p_n^{k+\ell -m*}- (\ell/n) \p_n^{n-m^*})} \\
    &\le \frac{3\abs{a_n}}{1-\p_n} \ell^{*(1-\a)}.
  \end{align*}

  Now let us turn to second configuration $I_2 := \{k,\ell : [k+1,
  k+\ell] \subset [k^*+1,m^*]\}$. Obviously
  \begin{align*}
    \sum_{j=k+1}^{k+\ell} \tau_{n,j-1} &= \frac{a_n}{1-\p_n}\left(\ell
      - \frac{\p_n^{k-k^*}(1-\p_n^{\ell})}{1-\p_n}\right),
  \end{align*}
  so
  \begin{align*}
    T_{\a,n}^{(2)} &= \frac{\abs{a_n}}{1-\p_n}  \max_{k,\ell \in I_2} \ell^{-\a} \abs{\ell -  \frac{\p_n^{k-k^*}(1-\p_n^{\ell})}{1-\p_n} - \frac{\ell}{n} \left(\ell^* - \p_n^{n-m^*} \frac{1-\p_n^{\ell^*}}{1-\p_n}\right)} \\
    &\le \frac{4\abs{a_n}}{1-\p_n} \ell^{*(1-\a)}.
  \end{align*}

  If we consider the third configuration $I_3 := \{k,\ell : k+1 <
  k^*+1 \le k+\ell < m^* \}$, we have
  \begin{align*}
    \sum_{j=k+1}^{k+\ell} \tau_{n,j-1} &= a_n \sum_{j=k+1}^{k+\ell} \sum_{i=1}^{j-1} \p_n^{j-1-i} \ind_{\I_n^*}(i) = a_n \sum_{j=k^*+1}^{k+\ell} \sum_{i=k^*+1}^{j-1} \p_n^{j-1-i} \\
    &= \frac{a_n}{1-\p_n} \left((k+\ell-k^*) -
      \frac{1-\p_n^{k+\ell-k^*}}{1-\p_n}\right).
  \end{align*}
  Since $k+\ell-k^* \le \ell^*$, then it is easy to see, that
  \begin{align*}
    T_{\a,n}^{(3)} &= \frac{\abs{a_n}}{1-\p_n}  \max_{k,\ell \in I_3} \ell^{-\a} \abs{(k+\ell-k^*) - \frac{1-\p_n^{k+\ell-k^*}}{1-\p_n} - \frac{\ell}{n} \left(\ell^* - \p_n^{n-m^*} \frac{1-\p_n^{\ell^*}}{1-\p_n}\right)} \\
    &\le \frac{4\abs{a_n}}{1-\p_n} \ell^{*(1-\a)}.
  \end{align*}

  Next, fourth configuration is $I_4 := \{k,\ell : k^*+1 < k+1 \le m^*
  < k+\ell \}$. Now
  \begin{align*}
    \sum_{j=k+1}^{k+\ell} \tau_{n,j-1} &= \frac{a_n}{1-\p_n}
    \left[(m^*-k) - \p_n^{k-k^*} \frac{1-\p_n^{m^*-k}}{1-\p_n} +
      (1-\p_n^{k+\ell-m^*}) \frac{1-\p_n^{\ell^*}}{1-\p_n}\right]
  \end{align*}
  together with \eqref{170413b} and $m^*-k \le \ell^*$ gives the
  estimate
  \begin{align*}
    T_{\a,n}^{(4)} &= \frac{\abs{a_n}}{1-\p_n}  \max_{k,\ell \in I_4} \ell^{-\a} \Big|(m^*-k) - \p_n^{k-k^*} \frac{1-\p_n^{m^*-k}}{1-\p_n} + (1-\p_n^{k+\ell-m^*}) \frac{1-\p_n^{\ell^*}}{1-\p_n} \\
    &\quad{} - \frac{\ell}{n} \left(\ell^* - \p_n^{n-m^*}
      \frac{1-\p_n^{\ell^*}}{1-\p_n}\right)\Big| \le
    \frac{5\abs{a_n}}{1-\p_n} \ell^{*(1-\a)}.
  \end{align*}

  From the fifth configuration $I_5 := \{k,\ell : m^* < k+1 < k+\ell
  \}$, we get
  \begin{align*}
    \sum_{j=k+1}^{k+\ell} \tau_{n,j-1} &= \frac{a_n}{1-\p_n} \cdot
    \p_n^{k-m^*} \frac{(1-\p_n^{\ell})(1-\p_n^{\ell^*})}{1-\p_n}
  \end{align*}
  and together with \eqref{170413b} the estimate is
  \begin{align*}
    T_{\a,n}^{(5)} &= \frac{\abs{a_n}}{1-\p_n}  \max_{k,\ell \in I_5} \ell^{-\a} \abs{\p_n^{k-m^*} \frac{(1-\p_n^{\ell})(1-\p_n^{\ell^*})}{1-\p_n} - \frac{\ell}{n} \left(\ell^* - \p_n^{n-m^*} \frac{1-\p_n^{\ell^*}}{1-\p_n}\right)} \\
    &\le \frac{3\abs{a_n}}{1-\p_n} \ell^{*(1-\a)}.
  \end{align*}

  Finally sixth configuration $I_6 := \{k,\ell : k+1 < k+\ell \le k^*
  \}$ gives us
  \begin{align*}
    \sum_{j=k+1}^{k+\ell} \tau_{n,j-1} &= 0.
  \end{align*}
  Thus
  \begin{align*}
    T_{\a,n}^{(6)} &= \frac{\abs{a_n}}{1-\p_n} \abs{\ell^* -
      \p_n^{n-m^*} \frac{1-\p_n^{\ell^*}}{1-\p_n}} \max_{k,\ell \in
      I_6} \ell^{-\a} \frac{\ell}{n} \le \frac{2\abs{a_n}}{1-\p_n}
    \ell^{*(1-\a)}.
  \end{align*}

  So collecting all the estimates of $T_{\a,n}^{(v)}$, $v=1, \ldots ,
  6$ we obtain \eqref{proo170413A}.

\newpage 
\bibliographystyle{gSTA}
\bibliography{Jurgita}

\end{document}